\documentclass[reqno]{amsart}

\usepackage[sc]{mathpazo}
\usepackage{eulervm}
\usepackage{ textcomp }

\usepackage[T1]{fontenc}

\usepackage{graphicx,mathrsfs,amssymb}
\usepackage{bbm}
\usepackage[utf8]{inputenc}
\usepackage{hyperref}
\usepackage{cleveref}
 \usepackage{tikz}

\usepackage[english]{babel}
\usepackage{verbatim}
\usepackage{comment}
\usepackage{todonotes}

\usepackage{enumerate}

\usepackage{fancyhdr}
\makeatletter
\let\ps@plain\ps@fancy
\makeatother
\newtheorem{theorem}{Theorem}[section]
\newtheorem{lemma}[theorem]{Lemma}
\newtheorem{proposition}[theorem]{Proposition}
\newtheorem{corollary}[theorem]{Corollary}

\theoremstyle{definition}
\newtheorem{definition}[theorem]{Definition}

\theoremstyle{remark}
\newtheorem{remark}[theorem]{Remark}
\newtheorem{example}[theorem]{Example}

\theoremstyle{Conjecture/open problem}

\def\op{\textrm{op}}

\def\R{\mathbb{R}}

\renewcommand{\AA}{{\mathscr A}}
\def\FF{{\mathscr F}}

\def\LL{\mathscr L}
\def\deg{\operatorname{deg}}
\def\BB{\mathscr B}
\def\LIP{\operatorname{LIP}}
\def\rk{\operatorname{rk}}

\def\deg{\operatorname{deg}}
\def\dim{\operatorname{dim}}

\begin{document}

\title{Fundamental polytopes of metric trees \\ via parallel connections of matroids}
\author[Emanuele Delucchi]{Emanuele Delucchi}
\address{(Emanuele Delucchi) Department of Mathematics, University of Fribourg, Chemin du Mus\'ee 23, CH-1700, Fribourg, CH.}
\email{emanuele.delucchi@unifr.ch}

\author[Linard Hoessly]{Linard Hoessly}
\address{(Linard Hoessly) Department of Mathematics, University of Fribourg, Chemin du Mus\'ee 23, CH-1700, Fribourg, CH.}
\email{linard.hoessly@unifr.ch}
\date{}

\subjclass{05-XX, 92Bxx, 54E35}
\keywords{Arrangements of hyperplanes, matroids, finite metric spaces, polytopes}

\begin{abstract}

We tackle the problem of a combinatorial classification of finite metric spaces via their {\em fundamental polytopes}, as suggested by Vershik in 2010 \cite{Vershik1}.

In this paper we consider a hyperplane arrangement associated to every split pseudometric and, for tree-like metrics,  we study the combinatorics of its underlying matroid. 
\begin{itemize}
\item
We give explicit formulas for the face numbers of fundamental polytopes and Lipschitz polytopes of all tree-like metrics, 
\item 
We characterize the metric trees for which the fundamental polytope is simplicial. 
%In particular, we obtain that no tree-like metric is generic in the sense of Gordon and Petrov.
\end{itemize}
\end{abstract}
\maketitle

\section{Introduction}

\subsection{Polytopes associated to metric spaces}
The study of {\em fundamental polytopes} of finite metric spaces was proposed by Vershik \cite{Vershik1} as an approach to a combinatorial classification of metric spaces, motivated by its connections to the transportation problem. Indeed, the Kantorovich-Rubinstein norm associated to the finite case of the transportation problem is an extension (uniquely determined by some conditions) of the Minkowski-Banach norm associated to the fundamental polytope (see \cite[Theorem 1]{Vershik2} for details). 
The polar dual of the fundamental polytope affords a more direct description: it consists of all real-valued functions with Lipschitz constant $1$, and it is called {\em Lipschitz polytope}. 
As polar duality preserves all combinatorial data, the combinatorial classification of Lipschitz polytopes is equivalent to that of fundamental polytopes.

Very little is known to date about the combinatorics of these polytopes, aside from the aforementioned work of Vershik. For instance, their {$f$-vectors}\footnote{The {\em $f$-vector} of a polytope (or of any polyhedral complex) is the list of integers encoding the number of faces of each dimension.}  are unknown in general. Gordon and Petrov \cite{Petrov} obtained  bounds for the number of possible different $f$-vectors given the size of the metric space. The same authors also examined ``generic metric spaces'' (see Definition \ref{df:generic}), computing their $f$-vectors (which, in this class, only depend on the number of elements in the space). 
Further study of fundamental polytopes of generic metrics appeared in \cite{JJZ,JTZ} especially around a connection with duals of cyclohedra and Bier spheres. The special case of fundamental polytopes of full trees (Definition \ref{df:tm}) fits into the framework of symmetric edge polytopes, which are themselves in the focus of active research, see e.g.\ \cite{Giapponesi,Michalek}.

In this paper we compute the $f$-vectors of Lipschitz polytopes for all tree-like pseudometric spaces, hence also of fundamental polytopes of tree-like metric spaces. Moreover, we characterize exactly which metric trees give rise fundamental polytopes that are simplicial. In particular, this characterization shows that our computations do not fall under the case considered in \cite{Petrov}.

\subsection{Arrangements of hyperplanes and matroids} We call "arrangement of hyperplanes" a finite set of hyperplanes (i.e., linear codimension 1 subspaces) of a real vectorspace and refer to Section \ref{ss:AH} for some basics about these well-studied objects. Here we only point out that the enumerative combinatorics of an arrangement is governed by the associated {\em matroid}, an abstract combinatorial object encoding the intersection pattern of the hyperplanes (see Section \ref{ss:mat}). 

In particular, such an arrangement subdivides the unit sphere into a polyhedral complex $K_{\AA}$ which is ``combinatorially dual'' (see Remark \ref{rm:az}) to the zonotopes arising as Minkowski sum of any choice of normal vectors for the hyperplanes. The enumeration of the faces of these polyhedral complexes in terms of the arrangement's matroid, due to Thomas Zaslavsky \cite{Zaslavsky}, has been one of the earliest successful applications of matroid theory. More recently, Cuntz and Gies \cite{CG} have given a necessary and sufficient condition for $K_\AA$ to be a simplicial complex, again in terms of enumerative invariants of the associated matroid.

The idea of this paper is to notice that fundamental polytopes and Lipschitz polytopes of finite metric trees are intimely related to the complex $K_\AA$ of an arrangement that is canonically associated to the metric space, and then to exploit the fact that the arrangement's matroid has a nice decomposition as a parallel connection of simple sub-matroids.

\subsection{Structure of the paper and main results} We start Section \ref{sec_mc} by recalling the main definitions and some results about polytopes associated to metric spaces, tree-like metrics, systems of splits, arrangements of hyperplanes and matroids. In particular, we focus on an arrangement of hyperplanes $\AA(\mathcal S)$ that can be associated to any system of splits $\mathcal S$. This arrangement and the associated matroid $\mathscr M(\mathcal S)$ (which were already considered in a different context \cite{Koichi}), provide the combinatorial underpinning of our considerations. 

\begin{itemize}
\item[(1)] We notice that the fundamental polytope of any tree-like finite metric space is combinatorially isomorphic to the complex $K_{\AA(\mathcal S)}$, where $\mathcal S$ is the unique system of splits in the Bandelt-Dress decomposition of the given space. This is because the Lipschitz polytope of such spaces is the zonotope defined as the Minkowski sum of a certain choice of normal vectors for the hyperplanes of $\AA(\mathcal S)$: this is the content of Theorem \ref{Lipschitz}, see also Remark \ref{rem_mago}.
\item[(2)] We compute the intersection poset of $\AA(\mathcal S)$ (and the closure operator of the associated matroid) from the combinatorics of the split system (Theorem \ref{intersection}). 
\item[(3)] We prove that the matroid of $\AA(\mathcal S)$ decomposes as a ``parallel connection'' of elementary building blocks that can be read off the (unique) metric tree representing the metric at hand. This allows us to give explicit formulas for the face numbers of fundamental- and Lipschitz polytopes in terms of the combinatorial structure of the tree (Theorem \ref{theo:formel}), building on Zaslavsky's theorems and on results of Bonin and De Mier on characteristic polynomials of parallel connections.
\item[(4)] Our formulas allow us to use Cuntz and Gies' criterion for simpliciality of arrangements in order to prove that the fundamental polytope of a tree-like metric is simplicial if and only if every vertex of the metric tree is associated an element of the metric space (Theorem \ref{theo:simp}). 

Our characterization shows in particular that no tree-like metric is generic in the sense of Gordon and Petrov \cite{Petrov} (Corollary \ref{cor:generic}).
\end{itemize}

\subsection{Related work} The study of metric spaces  by means of associated polyhedral complexes is a classical topic, going back at least to work of Buneman \cite{Buneman} and Bandelt-Dress \cite{Dress}, and driven in part by application to the study of phylogenetic trees. After a first version of this paper circulated, we learned about further recent literature that helped us to contextualize our work. Koichi \cite{Koichi} recently gave a uniform description of the approaches by Buneman and Bandelt-Dress, building on Hirai's \cite{Hirai} polyhedral split decomposition method, where a metric is viewed as a polyhedral "height function" defined on a point configuration. In \cite{Koichi} we also find the defining forms of the arrangement $\AA(\mathcal S)$. (On the other hand, the hyperplanes associated to splits in \cite[p.\ 350]{Hirai} do not coincide with ours.) 
Motivated by the connections to tropical convexity \cite{DeStu,StuYu}, Herrmann and Joswig \cite{HJ} studied split complexes of general polytopes and, in the process, consider an arrangement of ``split hyperplanes'' associated to every split metric. In this respect we notice that, even if each of our hyperplanes  can be expressed in the form \cite[Equation (9)]{HJ}, the arrangement $\AA(\mathcal S)$ is not one of the arrangements considered in \cite{HJ} (see Remark \ref{rem:HJ}).  Moreover, the matroid we consider is different from the matroid whose basis polytope is cut from the hypersimplex by a set of compatible split hyperplanes, which is studied by Joswig and Schr\"oter in  \cite{JS}.

Lipschitz polytopes of finite metric spaces are weighted digraph polyhedra in the sense of Joswig and Loho \cite{JL}, who give some general results about dimension, face structure and projections \cite[\S 2.1, 2.2, 2.6]{JL} but mostly focus on the case of ``braid cones'' which does not apply to our context. We close by mentioning that the polyhedra  considered, e.g., in the above-mentioned work of Hirai \cite[Formula (4.1)]{Hirai} are different from the Lipschitz polytopes we consider here: in fact, such polyhedra are (translated) zonotopes for all split-decomposable metrics \cite[Remark 4.10]{Hirai}, while -- for instance --  the Lipschitz polytope of any split-decomposable  metric on $4$ points is only a zonotope if  the associated split system is compatible.

\subsection{Acknowledgements} 
We are greatly indebted to Joseph Bonin for pointing out the decomposition of $\mathscr M(\mathcal S)$ as a parallel connection of elementary matroids. 
%We acknowledge the friendly hospitality of the department of Mathematics at the University of Fribourg.
 We thank Andreas Dress  for a friendly e-mail exchange and Yaokun Wu for discussing an announcement of his joint work with Zeying Xu on metric spaces with zonotopal Lipschitz polytopes, as well as for pointing out \cite{Wu}. 
After a first version of this paper was put on ArXiv, we received many valuable pointers to relevant extant literature: we thank two anonymous referees for their comments, Mateusz Micha\l{}ek for pointing out \cite{Michalek,Giapponesi} as well as Michael Joswig and Benjamin Schr\"oter for very informative discussions that took place during the program on tropical geometry at the institute Mittag-Leffler, whose excellent hospitality we also acknowledge.  

Both authors are supported by the Swiss National Science Foundation Professorship grant PP00P2\_150552/1.

\section{The main characters}\label{sec_mc}
\subsection{Metric spaces and their polytopes}

\begin{definition}\label{df:metric} Let $X$ be a set. A {\em metric} on $X$ is a symmetric function $d: X\times X \to \mathbb R_{\geq 0} $ with the following properties. 
\begin{itemize}
\item[(1)] For all $x,y\in X$, $d(x,y)=0$ implies $x=y$.
\item[(2)] For all $x,y,z\in X$, $d(x,y) + d(y,z) \geq d ( x,z)$ (``triangle inequality").
\end{itemize}
If requirement (1) is dropped, then $d$ is called a {\em pseudometric}. 
The pair $(X,d)$ is then called a metric space (resp.\ pseudometric space).\end{definition}

In this paper we will focus on {\em finite} metric spaces, i.e., metric spaces $(X,d)$ where $\vert X \vert <\infty$. We will tacitly assume so throughout.

\begin{definition}\label{df:poly} Let $(X,d)$ be a (finite) metric space. Consider the vectorspace $\mathbb R^X$ with its standard basis $\{\mathbbm{1}_k\}_{k\in X}$, i.e.,
$$(\mathbbm{1}_{k})_i:=\begin{cases} 1 & \mbox{if }i = k  \\
0 & \mbox{otherwise. } \end{cases} $$

Following Vershik \cite{Vershik1} we define the {\em fundamental polytope} of $(X,d)$ as 
$$P_d(X):=\operatorname{conv}\{ e_{i,j} \mid i,j\in X,\, i\neq j\},$$ 
where $$e_{i,j}:=\frac{\mathbbm{1}_i-\mathbbm{1}_j }{ d(i,j)}.$$

This polytope is contained (and full-dimensional) in the subspace
$$
V_0(X) = \{x\in \mathbb{R}^X \mid \textstyle{\sum}_i{x_i}=0\}.
$$

\end{definition}

\begin{definition} Let $(X,d)$ be a (finite) pseudometric space.

The {\em Lipschitz polytope}  of $(X,d)$ is given as an intersection of halfspaces by
\begin{equation}\label{df_lip}\LIP(X,d):=\left\{x\in \mathbb{R}^X \mid \textstyle{\sum}_i{x_i}=0, \, x_i-x_j\leq d(i,j) \,\,\forall i,j \in X\right\}.\end{equation}

This polytope is contained (and full-dimensional) in the subspace
$$V(X,d):=\{x\in \mathbb{R}^X  \mid \textstyle{\sum}_i{x_i}=0,\,\,
x_i=x_j\textrm{ whenever }d(i,j)=0\}.$$

\end{definition}

\begin{remark}[On Lipschitz polytopes] For  metric spaces our definition specializes to the standard definition of the Lipschitz polytope, e.g., as given in \cite{Vershik1}. We remark that, although related, this is not the set of Lipschitz functions considered in the work of Wu, Xu and Zhu on graph indexed random walks \cite{Wu}. 
\end{remark}

\begin{remark}[On polytopes] We point the reader to the book by Ziegler \cite{Ziegler} for terminology and basic facts about polytopes and fans. Here let us only mention that the combinatorics of a given polytope $P$ is encoded in its {\em poset of faces} $\FF(P)$ which, here, we take to be the set of all faces of $P$ including the empty face ordered by inclusion. A rougher, but very important enumerative invariant of a polytope are its {\em face numbers} $f_0^P,\ldots,f_{\dim(P)}^P$, where
$$
f_i^P = \vert \{i-\textrm{dimensional faces of }P\} \vert.
$$
It is customary to consider the empty face as a face of ``dimension $-1$'',  thus to write $f_{-1}^P = 1$ and to fit these numbers into the {\em $f$-polynomial} of $P$, defined as
$$f^P(t):= f_{-1}^P t^{m+1} + f_{0}^P t^{m} + \ldots +  f_{m}^P $$
where we write $m:=\dim(P)$. 
\end{remark}

The problem posed by Vershik \cite{Vershik1} is to study the face numbers and face structure of the fundamental polytope of a metric space. We will do so by focussing on the associated Lipschitz polytope, whose combinatorics is ``dual'' to that of the fundamental polytope in the following sense.

\begin{remark}\label{rem:lp} A look at Theorem 2.11.(vi) of \cite{Ziegler} shows that indeed, for every {\em metric} space $(X,d)$ the polytopes $P_d(X)$ and $\LIP(X,d)$ are {\em polar dual} to each other (with respect to the ambient space $V_0(X)$, cf.\ \cite[Definition 2.10]{Ziegler}). Polar duality induces an isomorphism of posets
$
\FF(P_d(X)) \cong \FF(\LIP(X,d))^{op}
$
and, in particular, the equality $f_i^{P_d(X)} = f_{\dim(P_d(X))-1-i}^{\LIP(X,d)}$.
\end{remark}

\begin{example}[Metric spaces from weighted graphs]  
Let $G$ be a finite, connected and simple graph. Write $V(G)$ and $E(G)$ for the set of vertices, resp.\ edges of $G$.

 A {\em weighting} of $G$ is any function $w: E(G) \to \mathbb R_{>0}$, and the pair $(G,w)$ is called a {\em weighted graph}. Then, setting
$$
d_w(v,v'):=\min \left\{w(e_1)+\ldots+ w(e_k) \mid e_1,\ldots,e_k \textrm{ an edge-path joining }v\textrm{ with } v'\right\}
$$
the pair $(V(G),d_w)$ is a metric space.
\end{example}

Recall that a {\em tree} is a graph in which every pair of vertices is connected by a unique path. 

\begin{definition}[Tree-like metrics and $X$-trees]\label{df:tm}
Let $X$ be a finite set. An {\em $X$-tree} is a pair $(T,\phi)$, where $T$ is a tree and $\phi:X\to V(T)$ is a map whose image contains every vertex of $V$ that is incident to at most two edges, i.e., $\{v\in V(T) \mid \deg(v) \leq 2 \} \subseteq \phi(X) $.

A (pseudo)metric $d$ on a set $X$
is called a {\em tree-like (pseudo)metric} if there exists an $X$-tree $(T,\phi)$ and a weighting $w$ of $T$ such that for all $x,y\in X$
$$d(x,y)=d_w(\phi(x),\phi(y)).$$
(``$d$ is induced by a weighted $X$-tree''). The pseudometric $d$ is a metric if and only if $\phi$ is injective. When $\phi$ is bijective, we call $(X,d)$ a {\em full tree}.
\end{definition}

\subsection{Arrangements of hyperplanes}\label{ss:AH}

Let $V$ denote a finite-dimensional real vectorspace, say of dimension $m$. 
An {arrangement of hyperplanes} (or, for short, {\em arrangement}) in $V$, is a finite set $\AA$
of hyperplanes (i.e., linear subspaces of codimension $1$). Such an arrangement defines a polyhedral fan in $V$, and we let $\FF(\AA)$ denote the poset of all faces of this fan, partially ordered by inclusion. %Here and in the following w 
We write $f^{\AA}_i$ for the number of faces of this fan of dimension $i$, for all $i=0,\ldots,m$,
and we arrange these numbers into the {\em $f$-polynomial} of $\AA$,
$$
f^{\AA}(t) := f^{\AA}_0t^m + f^{\AA}_1t^{m-1}+\ldots + f^{\AA}_m.
$$

The {\em poset of intersections} of $\AA$ is the set
$$
\LL(\AA) :=\{\cap \BB \mid \BB\subseteq \AA \},\quad x\leq y \Leftrightarrow x\supseteq y
$$
of all subspaces that arise as intersections of hyperplanes in $\AA$, partially ordered by reverse inclusion. 
The poset $\LL(\AA)$ is {\em ranked} by the function 
$
\rk(x):= m - \dim(x).
$
and we define the {\em rank of $\AA$} to be $r:=\rk(\cap \AA)$. 
The {\em M\"obius polynomial} of $\AA$ is
\begin{equation}\label{eq:MP}
M_{\AA}(u,v):=\sum_{x, y \in \LL(\AA)} \mu(x,y)u^{\rk(x)}v^{r-\rk(y)}
\end{equation}
where $\mu$ denotes the M\"obius function of $\LL(\AA)$ (see e.g.\ \cite[(3.17)]{Stanley}). 

\begin{theorem}[Zaslavsky {\cite[Theorem A]{Zaslavsky}}]\label{thm_z}
$$f^{\AA}(x)=(-1)^{r}\, M_{\AA}(-x,-1).$$
\end{theorem}

\subsection{Zonotopes} Associated to every set of nonzero real vectors $v_1,\ldots, v_k \in \mathbb R^m \setminus\{0\}$ there is a polytope obtained as the Minkowski (i.e., pointwise) sum
$$
Z(v_1,\ldots,v_k):=\sum_{i=1}^k [-1,1]v_i
$$
 where $[-1,1] \subseteq \mathbb R$ denotes the $1$-dimensional unit cube
 (see \cite[\S 1.1]{Ziegler}). Polytopes of this form are called {\em zonotopes}.
Strongly related to  
$Z(v_1,\ldots,v_k)$ is the arrangement of normal hyperplanes to the $v_i$, i.e., $\AA:=\{v_i^\perp \mid i=1,\ldots,k\}$. In particular, there is an isomorphism of posets (see, e.g., \cite[Corollary 7.18]{Ziegler})
$$
\FF(\AA)^{\op} \cong \FF(Z(v_1,\ldots, v_k))\setminus \{\emptyset\}
$$
which implies the following relationship among the $f$-polynomials.
\begin{equation}\label{eq_ZA}
f^{Z(v_1,\ldots,v_k)}(t) - t^{r+1}= t^{2m-r} f^{\AA}(\frac{1}{t})
\end{equation}

\subsection{Split systems} 
We introduce a special class of pseudometric spaces, keeping the terminology that is in use in the literature (e.g., \cite{Dress, Buneman}).
\begin{definition}\label{def_split_first}
Let $X$ be a finite set. A {\em split} of $X$ is a bipartition of $X$, i.e., a pair of nonempty and disjoint subsets $A,B\subseteq X $ (the {\em sides} of the split) such that $A\cup B = X$. Such a pair will be written $A|B$. Clearly, $A\vert B$ and $B\vert A$ describe the same split.
In fact, every split $\sigma=A\vert B$ corresponds to a nontrivial equivalence relation $\sim_\sigma$ on $X$, whose equivalence classes are $A$ and $B$. Given a split $\sigma$ and any element $i\in X$  we write $[i]_\sigma$ for the equivalence class of $i$ with respect to the equivalence relation $\sim_\sigma$. Thus, to any split $\sigma$ we can 
associate the function
\begin{equation}\label{delta}
\delta_{\sigma}(i,j)=\begin{cases} 0 &i \sim_\sigma j  \\
1 & \mbox{otherwise.} \end{cases} 
\end{equation}

A split $\sigma$ is called {\em trivial} if one of its sides is a singleton. We will use the shorthand $\sigma=k\vert k^c$ in order to denote a trivial split whose singleton side is $\{k\}$.

Two splits $A\vert B$ and $C\vert D$ are {\em compatible} if at least one of the sets $A\cap C$, $ A\cap D$, $B\cap C$, $B\cap D$ is empty.

A {\em system of splits} on $X$ is just a set of splits of $X$; the system is called \label{compatible}{\em compatible} if its elements are pairwise compatible.

\end{definition}

\begin{definition}\label{df_sww} A {\em weighted split system} is a pair $(\mathcal{S},\alpha)$ where $\mathcal{S}$ is a system of splits on $X$ and  $\alpha\in (\mathbb R_{\geq 0})^{\mathcal{S}}$ is any weighting.   Any such weighted split system defines a symmetric nonnegative function $d_{\alpha}: X\times X \to \mathbb R$ via $$d_\alpha(x,y)=\sum_{\sigma\in \mathcal{S}} \alpha_\sigma \delta_{\sigma}(x,y)$$ where $\delta_\sigma$ is as in Equation \eqref{delta}. The functions of the form $d_\alpha$ are called {\em split-decomposable pseudometrics} associated to $\mathcal{S}$. In fact, the pair $(X,d_\alpha)$ is a pseudometric space. We will write
$$
V(\mathcal S) := V(X,d_\alpha)
$$
as this subspace clearly does not depend on $\alpha$.
 A {\em positively weighted} split system is one where $\alpha_\sigma > 0 $ for all $\sigma \in\mathcal{S}$.
\end{definition}
Such metric spaces are also known as {\em cut (pseudo)metrics} \cite{Deza}.

\begin{theorem}[See {\cite[Theorems 3.1.4, 7.1.8, 7.2.6, 7.3.2]{Steel}}]\label{tree} Let $(X,d)$ be a pseudometric space. The following are equivalent:
\begin{enumerate}[(i)]
\item $d$ satisfies the ``four point condition'': for all $x,y,z,w\in X$,
$$d(x,y)+d(z,w)\leq \max \vert\left\{d(x,z)+d(y,w),d(x,w)+d(z,y)\right\}\vert$$
\item $d$ is a tree-like pseudo-metric on $X$ (in the sense of Definition \ref{df:tm}).
\item $d$ is a split-decomposable pseudometric associated to a positively weighted system of compatible splits. Moreover, this system is unique.
\end{enumerate}

\end{theorem}

\begin{remark} \label{edges_tree}
Under the equivalence of (ii) with (iii), splits in the decomposition of the metric correspond bijectively to edges in the tree.\end{remark}

\subsection{Arrangements associated to split systems} We now define an arrangement of hyperplanes associated to any split system. This set of hyperplanes appeared already in \cite[p.\ 10]{Koichi}, see Remark \ref{rem:Koichi}.

\begin{definition}\label{df_AS}
Let $X$ be a finite set and consider a split $\sigma=A\vert B$ of $X$, where $
|X|=n$. To $\sigma$ we associate the line segment (one-dimensional polytope) 
$$S_\sigma := \operatorname{conv}\left\{\frac{|B|}{n}\cdot\mathbbm{1}_A-\frac{|A|}{n}\cdot\mathbbm{1}_B,
\frac{|A|}{n}\cdot\mathbbm{1}_B-\frac{|B|}{n}\cdot\mathbbm{1}_A \right\} \subseteq V(\mathcal S) \subseteq \mathbb R^X$$
where $\mathbbm{1}_A:=\sum_{x\in A}\mathbbm{1}_x$, as well as a hyperplane 
$H_\sigma := (S_{\sigma})^{\perp}.$

Accordingly, the hyperplane arrangement and the zonotope associated to $\mathcal {S}$ are
$$
\AA(\mathcal {S}):=\{H_\sigma \mid \sigma\in \mathcal{S}\};
\quad\quad
Z(\mathcal{S}):= \sum_{\sigma\in \mathcal {S}} S_\sigma.
$$
\end{definition}

\begin{remark}\label{rm:az} Both the arrangement $\AA(\mathcal S)$ and the zonotope $Z(\mathcal S)$ are full-rank, resp.\ full-dimensional, inside the natural ``ambient space'' $V(\mathcal S)$. 
\end{remark}

\begin{remark}\label{rem:HJ}
Each of our hyperplanes $H_{\sigma}$ has the form of an $(A,B,\mu)$-hyperplane as described in \cite[Equation 8]{HJ}, for $\mu= p \vert B\vert $ and $k=pn$ where $p$ is any positive integer. However, such values of $\mu,k$ are excluded in \cite{HJ}. 
\end{remark}

\subsection{Matroids}\label{ss:mat} The  abstract combinatorial objects on which our enumerative considerations rest are {\em matroids}. Technicalities about matroids will only be needed in few proofs, therefore we only give a partial review of the definitions and the terminology and simply refer to Oxley's textbook \cite{Oxley}.

Let $E$ be a finite set. A matroid $\mathscr M$ on $E$ can be given by a collection of subsets of $E$ that contains the full set $E$ and which has the structure of a geometric lattice  when partially ordered by inclusion. The elements of this collection are the {\em flats} of the matroid, and the poset of all flats ordered by inclusion is called $\mathscr L(\mathscr M)$. Since geometric lattices are ranked posets, for every $A\subseteq E$ we can define a {\em rank} $\rho_{\mathscr M}(A)$ as the poset rank of the smallest element of $\mathscr L (\mathscr M)$ that contains $A$.

The {\em characteristic polynomial} of a matroid $\mathscr M$ on the ground set $E$ is
$$
\chi_{\mathscr M} (t):=\sum_{A\subseteq E} (-1)^{\vert A \vert} t^{\rho(E)-\rho(A)}.
$$

The matroid is called {\em simple} if the minimal flat is the empty set and every minimal nonempty flat is a singleton set. In this case, the structure of $\mathscr L(\mathscr M)$ determines the matroid fully.

\def\FM{\mathscr F}

\begin{example} Let $k$ be a positive integer and $E$ any $k$-element set. The set of all subsets of $E$ is the set of flats of a matroid on $E$ that we denote by $\FM(k)$ and call the free matroid on $k$ elements. The set of all subsets of $E$ of cardinality other than $k-1$ is also the set of flats of a matroid: we denote this by $\mathscr C(k)$ and call it the $k$-cycle matroid.
(The reader familiar with matroid theory will recognize $\mathscr C(k)$ as the uniform matroid $U_{k,k-1}$,  and notice that $\FM(k)\simeq U_{k,k}$.)
The characteristic polynomials of those matroids have the following form.
\begin{equation}
\chi_{\mathscr F(k)}(t) = (t-1)^k,\quad\quad
\chi_{\mathscr C(k)}(t) = (-1)^{k-1}\sum_{i=1}^{k-1}(1-t)^{i} 
\end{equation}

\end{example}

\begin{example}  To every arrangement $\mathscr A$ of hyperplanes in the sense of Section \ref{ss:AH} above is associated a matroid on the ground set $\mathscr A$ by declaring any $\mathcal K \subseteq \mathscr A$ to be a flat if and only if there is a linear subspace $X$ of $V$ such that $\mathcal K$ is the set of all hyperplanes containing $X$. In particular, there is a poset isomorphism
$$
\mathscr L (\mathscr A) \to \mathscr L (\mathscr M);\quad 
X \mapsto  \{H\in \mathscr A \mid X\subseteq H\}
$$
and, for every $\mathcal K \subseteq \mathscr A$, $\rho_{\mathscr M}(\mathcal K)=\operatorname{codim}\cap\mathcal K$.

From Zaslavsky's Theorem \ref{thm_z} and elementary computations we see that
\begin{equation}\label{eq:sum:contr}
f_{i}^{\mathscr A} = (-1)^i 
\sum_{\substack{
\mathcal K\in \mathscr L (\mathscr M) 
\\
\rho(\mathscr A) - \rho(\mathcal K)=i
}}
\chi_{\mathscr M / \mathcal K} (-1)
\end{equation}
where $\mathscr M / \mathcal K$ denotes the {\em contraction} of the flat $\mathcal K$ (see \cite[Section 3.1]{Oxley}). 
\end{example}

We conclude this brief overview with two matroid operations.

\begin{definition}
Let $E_1$, $E_2$ be two disjoint finite sets, and for $i=1,2$ let $\mathscr M_i$, be a matroid on the ground set $E_i$ with lattice of flats $\mathscr L_i$. The {\em direct sum} $\mathscr M_1 \oplus \mathscr M_2$ is the matroid on $E_1\cup E_2$ whose flats are precisely the unions of flats of $\mathscr M_1$ and $\mathscr M_2$. In particular, there is an isomorphism of posets $
\mathscr L(\mathscr M_1\oplus \mathscr M_2) \simeq \mathscr L_1\times \mathscr L_2$.
\end{definition}
The characteristic polynomial of a direct sum decomposes as a product.
\begin{equation}
\chi_{\mathscr M_1 \oplus \mathscr M_2}(t)=
\chi_{\mathscr M_1}(t)\chi_{\mathscr M_2}(t)
\end{equation}

\begin{definition}\label{def:PC}
Let $E_1$, $E_2$ be two finite sets such that $E_1\cap E_2=\{e\}$ for some $e$. For $i=1,2$ let $\mathscr M_i$, be a matroid on the ground set $E_i$ with lattice of flats $\mathscr L_i$ and rank function $\rho_i$. If $\rho_1(e)=\rho_2(e)$, the {\em parallel connection} of $\mathscr M_1$ and $\mathscr M_2$ along $e$ is the matroid $\mathscr M_1 \oplus_e \mathscr M_2$ on the ground set $E_1\cup E_2$ whose flats are precisely the subsets of the form $F_1\cup F_2$ for $(F_1,F_2)\in \mathscr L_1\times \mathscr L_2$ and either $e\in F_1\cap F_2$ or $ e\not\in F_1\cup F_2$.
\end{definition}

Characteristic polynomials of parallel connections behave naturally, e.g., as in the following sample result which we state for later reference.
\begin{remark}[{\cite[Theorem 5]{BdM}}]\label{rem:221} In the setting of Definition \ref{def:PC}, if $\rho_{\mathscr M_1}(\{e\})=\rho_{\mathscr M_2}(\{e\})\neq 0$ and $F$ is any flat of $\mathscr M_1\oplus_e\mathscr M_2$, then 
$$
\chi_{(\mathscr M_1 \oplus_e \mathscr M_2) / F}(t)=
\frac{\chi_{\mathscr M_1/(F\cap E_1)}(t)\chi_{\mathscr M_2/(F\cap E_2)}(t)}
{(t-1)^{\vert \{e\}\setminus F\vert}}
$$
\end{remark}

\section{Lipschitz polytopes of compatible systems of splits}

\begin{theorem} \label{Lipschitz} Let $(X,d)$ be a tree-like pseudometric space. Then,
$$LIP(X,d)= \sum_{\sigma\in \mathcal{S}}\alpha_{\sigma}S_\sigma$$
where $(\mathcal {S},\alpha)$ is the unique weighted system of compatible splits of $X$ such that  $d=d_\alpha$ (cf.\ Theorem \ref{tree}).

 \end{theorem}
\begin{remark}\label{rem_mago} 
We thank an anonymous referee for pointing out to us that this theorem can be deduced from work of Koichi \cite{Koichi} and Murota's book on convex discrete analysis \cite{Murota}, as we explain in Proof A.
 For the benefit of the reader who might not be familiar with this apparatus, we also offer an elementary direct argument (Proof B below).
 \end{remark}

 \begin{proof}[Proof A] Following, e.g., Murota \cite{Murota}, we identify  any $C\subset \R^X$ via its ``indicator function'' $\delta_C:\R^X\to\{0,+\infty \}$, defined as
 \begin{equation}\label{delta}
\delta_{C}(x)=\begin{cases} 0 &x\in C  \\
\infty & \mbox{otherwise.} \end{cases} 
\end{equation}
\newcommand{\bolddot}[1]{#1^\bullet}
When $C$ is convex, \cite[Theorem 3.2 and (3.31)]{Murota} says that the ``conjugate''  (or Legendre-Fenchel transform) $\bolddot{\delta_C}$ of $\delta_C$ satisfies
$$\bolddot{\delta_C}(x)=\sup_{y\in C}x^Ty.$$
This expression allows for an explicit verification of the fact that for every Minkowski-sum decomposition $C=A+B$ of $C$ into convex sets $A$ and $B$ and for every $\alpha>0$  we have 
\begin{equation}\label{convsplit}
    \bolddot{\delta_C}(x)=
\bolddot{\delta_A}(x)+\bolddot{\delta_B}(x),\quad
\alpha\cdot\bolddot{\delta_C}(x)=\bolddot{\delta_{\alpha C}}(x)
\end{equation}

Moreover, following \cite{Koichi}, to any finite (pseudo)metric space $(X,d)$  one can associate the  finite vector configuration $K:=\{\mathbbm{1}_u-\mathbbm{1}_v\}_{u,v\in X}\subseteq \mathbb R^X$ and consider the  function $\overline{d}:\mathbb R^X\to \mathbb R$ defined as  the homogeneous convex closure \cite[\S 2.3]{Koichi} of the discrete function $K\to \mathbb R$, $(\mathbbm{1}_u-\mathbbm{1}_v)\mapsto d(u,v)$. An explicit expression for $\overline{d}$ is given in \cite[(3.1)]{Koichi}, and a direct check shows that in this case
\begin{equation}\label{dLIP}
\overline{d}(x)=\bolddot{\delta_{\LIP(X,d)}}.
\end{equation}
In particular, since for every split $\sigma$ of $X$ the function $\delta_\sigma$ from Equation \eqref{delta} is a pseudometric on $X$, we have
\begin{equation}\label{1split}
 \overline{\delta_{\sigma}}=\bolddot{\delta_{LIP(X,\delta_{\sigma})}}.
\end{equation}

When $(X,d)$ is a tree-like metric space with associated split system $(\mathcal{S},\alpha)$, \cite[Proposition 3.6]{Koichi} shows that
\begin{equation}\label{dsplit}
\overline{d} = \sum_{\sigma\in \mathcal S}\alpha_\sigma \overline{\delta_{\sigma}},
\end{equation}
where ${\delta_{\sigma}}$ is as in Equation \ref{dLIP} .
Summing up,  we can write
$$
\bolddot{\delta_{\LIP(X,d)}}
\stackrel{\eqref{dLIP}}{=}
\overline{d}
\stackrel{\eqref{dsplit}}{=} 
\sum_{\sigma\in \mathcal S}\alpha_\sigma \overline{\delta_{\sigma}}
\stackrel{\eqref{1split}}{=}
\sum_{\sigma\in \mathcal S}
\alpha_\sigma \bolddot{\delta_{\LIP(X,\delta_{\sigma})}}
\stackrel{\eqref{convsplit}}{=}
 \bolddot{\delta_{\left(\sum_{\sigma\in \mathcal S}\LIP(X,\alpha_\sigma\delta_{\sigma})\right)}}.$$

From the equality of the Legendre-Fenchel transforms of the indicator functions one then deduces equality of the polytopes, proving the claim.

 \end{proof}

\begin{proof}[Proof B] The proof is by induction on the cardinality of $\mathcal S$. If $\vert \mathcal S \vert =0$ there is nothing to prove.

Let then $\vert \mathcal S \vert >0$ and suppose that the theorem holds for all weighted systems of compatible splits of smaller cardinality. By Theorem \ref{tree}, to the space $(X,d)$ is associated a weighted $X$-tree $(T,\phi)$ in the sense of Definition \ref{df:tm}, and the corresponding tree metric can be expressed as a split metric with a split for every edge in the tree. The uniqueness part of Bandelt and Dress' decomposition theorem (\cite[Theorem 2]{Dress}) says that the associated split system must be $\mathcal S$. In particular, 
the tree $T$ has at least one edge, and thus at least one leaf vertex (i.e., a vertex incident to exactly one edge). 
Choose then such a leaf vertex, say $v$, and let $\sigma \in \mathcal S$ be the split corresponding to the unique edge incident to $v$. Then, 
$$\sigma = A \mid A^c \textrm{ with } A:=\phi^{-1}(v).$$
 Let $\mathcal S':=\mathcal S\setminus \{\sigma\}$ and let $(X,d')$ be the pseudometric space defined by $\mathcal S'$ and the appropriate restriction of $\alpha$. 
 Now notice that $d=d'+\alpha_\sigma \delta_{\sigma}$ and that, for all $i,j\in A$, we have $d'(i,j)=0$. The claim then follows by induction hypothesis applied to $\mathcal S'$ via the following identity.
$$\LIP(X,d'+\alpha_{\sigma} \delta_{\sigma} )= \LIP(X,d')+\alpha_{\sigma}S_\sigma.$$
The right-to-left containment is verified directly. In order to check the 
left-to-right containment we consider a point $x\in \LIP(X,d'+\alpha_{\sigma} \delta_{\sigma})$ and prove that it is contained in the right-hand side. The definition of the Lipschitz polytope implies immediately that, for all $i,j\in X$,
 $x_i-x_j \leq d'(i,j)+\alpha_\sigma$.

 Define $$\alpha:=max_{i\in A, j\in A^c}\{0, x_i-x_j -d'(i,j), x_j-x_i - d(i,j)\}.$$

If $\alpha=0$, then $x\in \LIP(X,d')$. Otherwise, choose $i_0,j_0$ such that $\alpha = x_{i_0}-x_{j_0} - d(i_0,j_0)$. Assume w.l.o.g.\ $i_0\in A$ and $j_0\in A^c$ (otherwise switch $A$ and $A^c$ in the following). 
Since now $0\leq \alpha \leq \alpha_\sigma$, it is enough to show that
$$y:=x - \alpha v_{\sigma}\in \LIP(X,d')$$
where $v_\sigma:=\frac{|A^c|}{n}\cdot\mathbbm{1}_A-\frac{|A|}{n}\cdot\mathbbm{1}_{A^c}$. This is proved by verifying, with a direct computation, that $y$ satisfies Equation \ref{df_lip}.
\end{proof}

\begin{theorem}\label{thm_formula} Let $(X,d)$ be a tree-like pseudometric space with associated system of compatible splits $\mathcal S$. Then the $f$-vector of the associated Lipschitz polytope is as follows.
$$f^{\LIP(X,d)}(x)=(-x)^{\rk(\cap\AA(\mathcal{S}))} M_{\AA(\mathcal{S})}\left(-\frac{1}{x},-1\right) + x^{\rk(\cap\AA(\mathcal S))+1}$$
If additionally $(X,d)$ is a metric space, then the $f$-vector of the associated fundamental polytope is
$$f^{P_d(X)}(x)=(-1)^{\rk(\cap\AA(\mathcal{S}))} M_{\AA(\mathcal{S})}(-x,-1)x+1.$$
\end{theorem}

\begin{proof}
Theorem \ref{Lipschitz} implies that 
$
\FF(\LIP(X,d)) \simeq \FF(Z(\mathcal S))
$, and thus with Remark \ref{rm:az} Theorem \ref{thm_z} 
we can compute
$$
f^{\LIP(X,d)}(x)
=
(-x)^{\rk(\cap\AA(\mathcal{S}))} M_{\AA(\mathcal{S})}(-\frac{1}{x},-1)
+ x^{\rk(\cap\AA(\mathcal S))+1}
$$
This proves the first of the claimed equalities. The second follows by  duality (Remark \ref{rem:lp}).
\end{proof}

\begin{corollary}\label{eq:f:contr}
For any tree-like metric space $(X,d)$ 
$$
f_i^{P_d(X)} = f_{i+1}^{\AA(\mathcal S)} = f_{|X|-1-i}^{\LIP(X,d)}   
$$
where $\mathcal S$ denotes the associated system of (compatible) splits and the index $i$ runs from $0$ to $\dim(P_d(X))=|X|-1$.

\end{corollary}

\section{Computation of face numbers}

We turn to the problem of an effective computation of the $f$-vectors of fundamental polytopes. The main result of this section are explicit formulas for the face numbers of fundamental polytopes of tree-like metric spaces.

We will start with two easy cases and then offer a general tool allowing to compute the intersection lattice of the associated hyperplane arrangement. From there, we will study the structure of the matroid $\mathscr M(\mathcal S)$ in order to set up our formulas.

\begin{example}[Points in $\mathbb R^1$]\label{ex_points}
We can represent the metric space defined by any set of $n$ points in $\R^1$ by just taking its metric graph in a line, considering the associated set of splits and choosing the coefficients in the split-metric accordingly.  The arrangement corresponds to $(n-1)$ independent vectors in $n-1$-dimensional space, i.e., it is  isomorphic to the coordinate arrangement. The corresponding matroid is the uniform matroid $\mathcal{U}_{n-1}^{n-1}$ and, in particular, $f^{\AA}_i=2^{i}\binom{n-1}{i}$.
\end{example}
\begin{example}[The root polytope of type $A_{n-1}$]\label{ex_root} 
Let us consider a star graph, i.e., a tree with $n>2$ leaves and a unique internal vertex . If we assign each edge the length $\frac{1}{2}$, we define the structure of a metric space on the set $X$ of leaves of our star graph. 

The corresponding split system consists exactly of all the trivial splits, and  any two points are at distance $1$. Then, by definition, the fundamental polytope of this space is the convex hull of the vectors $e_{i,j}=\mathbb{1}_i-\mathbb{1}_j$, where $i\neq j \in [n]$.
This is also called the {\em root polytope of type $A_{n-1}$}, and its face numbers have been computed via algebraic-combinatorial considerations by Cellini and Marietti \cite[Proposition 4.3]{CelliniMarietti}. Of course, one could compute these numbers by computing the M\"obius function of the corresponding matroid, i.e., the uniform matroid $\mathcal{U}_{n}^{n-1}$.
\end{example}

\subsection{The intersection lattice of $\AA(\mathcal S)$} We will start by describing the intersection poset of $\AA(\mathcal S)$ by means of partitions. Work in this direction can be implicitly found in \cite{Koichi}, but for our purposes it will be convenient to give explicit statements and direct proofs (see Remark \ref{rem:Koichi} for details).

\begin{definition}\label{df_pi}
Let $(X,d)$ be a pseudometric space. The function $d$ induces a partition $\pi(d)$ of the set $X$ given as the set of equivalence classes of the equivalence relation in which $i$ and $ j$ are equivalent if and only if $d(i,j)=0$.
If the space $(X,d)$ arises from a positively weighted system of splits $(\mathcal S , \alpha)$, the partition $\pi(d)$ does not depend on $\alpha$ and we only write $\pi(\mathcal S)$.

 We have an order-reversing map of posets
$$
\pi: 2^{\mathcal S}\to \Pi_X; \quad\quad \mathcal S'\mapsto \pi(\mathcal S')
$$
where $2^{\mathcal S}$ denotes the poset of all subsets of $\mathcal S$ ordered by inclusion, and $\Pi_X$ is the poset of all partitions of $X$ ordered by refinement.

\end{definition}

\begin{theorem}\label{intersection}Let $(\mathcal {S},\alpha)$ be an arbitrary weighted system of compatible splits of a finite set $X$ and write $\pi:=\pi(\mathcal S)$. Then,
 $$\cap_{\sigma\in \mathcal{S}} H_{\sigma}= \big \langle e_{i,j} : i \textrm{ and }j \textrm{ in the same block of  } \pi(\mathcal S) 
% i\sim_{\pi} j 
 \big \rangle , $$
 where $e_{i,j}=(\mathbb{1}_i-\mathbb{1}_j)/d(i,j) $, see Definition \ref{df:poly}. 
\end{theorem}
\begin{proof}
The right-to-left inclusion holds by definition. We will prove the left-to-right inclusion by induction on the cardinality of the system of splits.

 If $\vert \mathcal S \vert = 1$ the claim is evident. Let then $m>0$, assume that the statement holds for any weighted system of up to $m$ compatible splits and consider a weighted system of splits $(\mathcal S,\alpha)$ with $\vert \mathcal S \vert = m+1$.

   By Theorem \ref{tree}, $(\mathcal S,\alpha)$ can be represented by an $X$-tree $(T,\phi)$ with at least one edge, hence with at least one leaf vertex $v$. In particular, with $A:=\phi^{-1}(v)$, we know that
   $
   \sigma:=A\vert A^c \in \mathcal S
   $
   and we can consider
   $$
   \mathcal S':= \mathcal S \setminus \{\sigma\},\quad\quad \alpha':= \alpha_{\vert \mathcal S'}, \quad\quad d':=d_{\alpha'}, \quad\quad \pi':=\pi(\mathcal S').
   $$

The $X$-tree $(T',\phi')$ associated to $(\mathcal S', \alpha')$ must have a vertex $v'$  with $\phi'(A)=v'$ (otherwise there would be $i,j\in A$ with $d_\alpha(i,j)\geq d_\alpha'(i,j)>0$). 

In a neighborhood of $v'$, the $X$-trees associated to $(\mathcal S',\alpha')$, resp.\ $(\mathcal S,\alpha)$,  differ as in Figure \ref{fig_v}. In particular, 
\begin{center}
($\ast$) $\quad$ $\pi = \{A,C,\pi_1,\ldots,\pi_k\}$; $\quad\quad$ $\pi'=\{A\sqcup C,\pi_1,\ldots,\pi_k \}$
\end{center}
where, as in the following, we think of a partition as a set of blocks. Moreover, given a partition $\pi$ of a set let $\sim_{\pi}$ denote the equivalence relation on the same set whose equivalence classes are the blocks of $\pi$.

\def\block{\pi}
\begin{figure}[h]
  \begin{tikzpicture}[scale=.8, every node/.style={scale=0.8}]
    \foreach \a in {0, 30, 90, 120, 150, 270, 300, 330}
      \draw (0, 0) -- (\a:2);
    % dots  
    \foreach \a in {45,50,55,60,65,70}
    \draw[fill=black] (\a: 2) circle(.1mm);
    %Leaves  
    \foreach \a in {0, 30, 90, 120, 150, 270, 300, 330}
%      \draw[fill=black] (\a: 2) circle(.8mm); %node {$\a^\circ$};
      \draw (\a: 2) node {}; %node {$\a^\circ$};      
    %Central point
    \draw[fill=black] (0,0) circle(.8mm);
    %Labels
    \draw (270: 2.4) node {}; %{}; %{$A_1$};
    \draw (300: 2.4) node {}; %{}; %{$A_2$};             
    \draw (330: 2.4) node {}; %{}; %{$A_3$};
    \draw (0: 2.4) node {}; %{}; %{$A_4$};
    \draw (30: 2.4) node {}; %{}; %{$A_5$};    
    \draw (90: 2.4) node {}; %{}; %{$A_{k-2}$};
    \draw (120: 2.4) node {}; %{}; %{$A_{k-1}$};
    \draw (150: 2.4) node {}; %{}; %{$A_{k}$};
    %% Extra point label
    \draw (210: .6) node {$A\sqcup C$};    
  \end{tikzpicture}$\quad\quad\quad\quad$
      \begin{tikzpicture}[scale=.8, every node/.style={scale=0.8}]
    %30ÃÂÃÂ° Rays
    \foreach \a in {0, 30, 90, 120, 150, 210, 270, 300, 330}
      \draw (0, 0) -- (\a:2);
    % dots  
    \foreach \a in {45,50,55,60,65,70}
    \draw[fill=black] (\a: 2) circle(.1mm);
    %Radius labels (background filled white)
 %   \foreach \r in {1, 2,...,7}
 %     \draw (\r,0) node[inner sep=1pt,below=3pt,rectangle,fill=white] {$\r$};
    %Leaves  
          \draw[fill=black] (210: 2) circle(.8mm); %node {$\a^\circ$};
    \foreach \a in {0, 30, 90, 120, 150, 210, 270, 300, 330}
%      \draw[fill=black] (\a: 2) circle(.8mm); %node {$\a^\circ$};
      \draw (\a: 2) node {}; %node {$\a^\circ$};
    %Central point
    \draw[fill=black] (0,0) circle(.8mm);
    %Labels
    \draw (270: 2.4) node {}; %{}; %{$A_1$};
    \draw (300: 2.4) node {}; %{}; %{$A_2$};             
    \draw (330: 2.4) node {}; %{}; %{$A_3$};
    \draw (0: 2.4) node {}; %{}; %{$A_4$};
    \draw (30: 2.4) node {}; %{}; %{$A_5$};    
    \draw (90: 2.4) node {}; %{}; %{$A_{k-2}$};
    \draw (120: 2.4) node {}; %{}; %{$A_{k-1}$};
    \draw (150: 2.4) node {}; %{}; %{$A_{k}$};
    %% Extra point label
    \draw (185: .5) node {$C$};    
    \draw (210: 2.5) node {$A$};    
  \end{tikzpicture}

  \caption{The neighborhood of the vertex $v'$ in the $X$-tree $T'$ (left-hand side) and $T$ (right-hand side).}\label{fig_v}
  \end{figure}
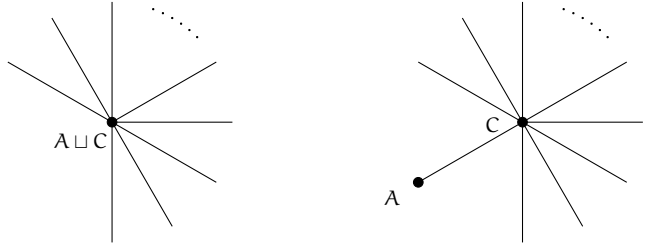

Let $v_{\sigma}:=\frac{|A|}{n}\cdot\mathbbm{1}_{A^c}-\frac{|A^c|}{n}\cdot\mathbbm{1}_{A}$, so that $(v_\sigma)^\perp = H_\sigma$. By induction hypothesis,
\begin{equation}\label{eq_Hvs}
\bigcap_{\tau\in \mathcal S} H_{\tau} = \bigcap_{\tau\in \mathcal S'} H_{\tau} \cap H_\sigma =\langle e_{i,j}: i\sim_{\pi'}j \rangle \cap (v_\sigma)^\perp.
\end{equation}

In view of $(\ast)$, the subspace $\langle e_{i,j}: i\sim_{\pi'}j \rangle$ decomposes as
$$
\bigoplus_{b\in \pi'}\left\langle e_{i,j}\mid i,j\in b \right \rangle =\left\langle e_{i,j}\mid i,j\in A\sqcup C \right \rangle\oplus
\underbrace{
\bigoplus_{b\in \pi'\setminus \{A\sqcup C\}}\left\langle e_{i,j}\mid i,j\in b \right \rangle}_{=:W}
$$

Since $\sigma$ does not split any block of $\pi'\setminus \{A\sqcup B\}$, we have $W\subseteq (v_\sigma)^{\perp}$. Therefore, the right-hand side of Equation \eqref{eq_Hvs} equals $\left(\left\langle e_{i,j}\mid i,j\in A\sqcup C \right \rangle \cap (v_\sigma)^\perp \right) \oplus W $.

On the other hand, 
$$\left\langle e_{i,j}\mid i,j\in A\sqcup C \right \rangle \cap (v_\sigma)^\perp
=\left\langle e_{i,j}\mid i,j\in A \right \rangle
\oplus
\left\langle e_{i,j}\mid i,j\in C \right \rangle$$

 Thus, we can rewrite the right-hand side of Equation \eqref{eq_Hvs} as
 $$
 \left\langle e_{i,j}\mid i,j\in A \right \rangle
\oplus
\left\langle e_{i,j}\mid i,j\in C \right \rangle
\oplus
W
 $$
 and in particular, recalling the block structure of $\pi$ from ($\ast$),
$$
\langle e_{i,j}\mid  i\sim_{\pi'} j \rangle \cap (v_\sigma)^\perp
= \langle e_{i,j}\mid  i\sim_{\pi} j \rangle 
$$
which, together with Equation \eqref{eq_Hvs}, concludes the proof.

\end{proof}

Recall that the posets of intersections of $\AA(\mathcal S)$ is the lattice of flats of the matroid $\mathscr M(\mathcal S)$. 

\begin{corollary}\label{cor:flats}
There is a poset isomorphism
$$
\mathscr{L} (\AA(\mathcal S)) \simeq \operatorname{im}\pi
$$
where the right-hand side is considered as an induced sub-poset of $\Pi_X^{op}$.

More precisely, if we identify the ground set of the matroid $\mathscr M(\mathcal S)$ with  $\mathcal S$ itself, we can write the closure operator of the matroid as
$$
\operatorname{cl (\mathcal S')} = \max_{\subseteq}\{ \mathcal S'' \subseteq \mathcal S \mid \pi(\mathcal S'')=\pi(\mathcal S'), \rho(\mathcal S') = \vert \pi(\mathcal S')\vert-1\}
$$
\end{corollary}

\begin{example}[Full trees]\label{ex:full}
If $(X,d)$ is a ``full tree'' (in the sense of Definition \ref{df:tm}), then it can be represented by an $X$-tree where each vertex is labeled by exactly one point of $X$. Therefore it is apparent that $\pi$ is injective, and thus the poset $\mathcal L (\AA (\mathcal S))$ is boolean. 
With equation \eqref{eq:sum:contr} and Corollary \ref{eq:f:contr} we immediately
$$
f_i^{P_d(X)} = 2^i{n-1 \choose i},
$$
generalizing, as expected, Example \ref{ex_points}. 
\end{example}

\begin{remark}\label{rem:Koichi} Koichi \cite{Koichi} considers lattices of flats $\mathscr L$ of matroids of linear dependencies of centrally symmetric vector configurations and characterizes the families $\mathcal C \subseteq \mathscr L\setminus (\max\mathscr L)$ 
such that the subposet $\{\wedge C\mid C\subseteq \mathcal C\}$ of $\mathscr L$ (i.e., generated by meets of subfamilies of $\mathcal C$) is anti-isomorphic to a geometric lattice \cite[Theorem 4.1]{Koichi}. Our case corresponds to point configurations of "type $\Omega$" in \cite{Koichi}, for which \cite[Theorem 4.6]{Koichi} establishes that the set of all flats consinsting of the vectors contained in one of the hyperplanes from Definition \ref{df_AS} satisfies indeed this condition. Moreover, in \cite[\S 4.3.1]{Koichi} it is hinted at a description of $\mathscr L(\AA(\mathcal S))$ in terms of partitions of $X$. We thought it helpful to give explicit statements and a direct proof in this paper.
\end{remark}

\subsection{A graph-theoretic description of $\mathscr M (\mathcal S)$}

In order to give explicit formulas for the face numbers of the fundamental polytope of a tree-like metric space $(X,d)$ we study the structure of the matroid $\mathscr M(\mathcal S)$ in terms of the associated $X$-tree.
First, note that since the ground set of $\mathscr M(\mathcal S)$ is the set of splits, via Theorem \ref{tree} we can naturally think of $\mathscr M(\mathcal S)$ as having the set of edges of the $X$-tree as a ground set.

\medskip

In this section let then $X$ be a finite set, and let $T$ denote any $X$-tree (Definition \ref{df:tm}). For simplicity we identify a labeled vertex with its label, thus regarding $X$ as a subset of $V(T)$. Given any $F\subseteq E(T)$, the associated {\em edge-induced subgraph} is $T[F]=(V(T), F)$, i.e., the graph consisting of all vertices of $T$ but only the edges in $F$. 

Every edge $e\in E(T)$ defines a split $\mathcal S_e$ of $X$ by partitioning $X$ into two parts according to which connected component of $T[E\setminus \{e\}]$ they are in. Let $\mathcal S(F):=\{\mathcal S_e \mid e\in F\}$ denote the system of splits associated to an edge set $F$.

\begin{definition}
Call a subset $F\subseteq E(T)$ of edges of an $X$-tree {\em flat} if the induced subgraph $T[E\setminus F]$ has no unlabeled vertices of degree $1$. 
\end{definition}

\begin{proposition}\label{prop:S-T} Let $T$ be an $X$-tree and let $\mathcal S$ denote the associated system of splits of $X$. Then, a set $F\subseteq E$ is flat if and only if $\mathcal S(F)$ is closed in $\mathscr M(S)$. Moreover, the rank of $\mathcal S(F)$ in $\mathcal M(T)$ is one less than the number of connected components of $T[E\setminus F]$ that contain at least an $X$-labeled vertex. 
\end{proposition}

\begin{proof}

The partition $\pi(\mathcal S(F))$ is the set of equivalence classes of the relation defined on $X$ by $x\simeq_F y$ if $x,y$ in the same connected component of $T[E\setminus F]$. Thus,  $\mathcal S(F)$ is not closed if and only if there is a split $\sigma$ in $\mathcal S \setminus \mathcal S(F)$ with $\pi(\mathcal S(F)\cup \{\sigma\})=\pi(\mathcal S(F))$. Equivalently, there is an edge $e\not\in F$ whose removal from $T[E\setminus F]$ does not increase the number of connected components containing $X$-vertices: this can only happen if  $T[E\setminus F]$ contains a non-$X$-labeled leaf.
\end{proof}

\begin{definition}
For any given $X$-tree $T$, let $\mathscr M(T)$ denote the unique simple matroid on the ground set $E(T)$ where a set is closed if and only if it is flat in $T$.
\end{definition}

\begin{corollary}
For every tree-like metric space with associated split-system $\mathcal S$ and tree $T$, the matroids $\mathscr M(\mathcal S)$ and $\mathscr M (T)$ are isomorphic.
\end{corollary}

In the following we will study how the structure of the tree $T$ leads to a decomposition of the matroid $\mathscr M(T)$. The decomposition is in terms of {\em parallel connections}. Recall that

\def\CT{\widehat{T}}

\begin{definition}\label{def:TC}
Given any tree $T$ let  $\CT$ be the tree obtained by removing al vertices of degree $1$ from $T$. We call $\CT$ the ``core'' of $T$.
  
For every leaf $c$ of $\CT$ let $\ell(c)$ be the set of leaves of $T$ adjacent to $c$. Let $c'$ be the vertex of $\CT$ adjacent to $c$. Then set $T_{[c]}$ to be the $(\ell(c)\cup\{c'\})$-tree obtained from the neighborhood of $c$ in $T$ by, if necessary, labeling $c'$. Moreover, let $T^{[c]}$ denote the $((X\setminus \ell(c))\cup \{c\})$-tree obtained by pruning $\ell(c)$ from $T$ and, if necessary, labeling $c$.
\end{definition}

\begin{lemma}\label{lemma:pruning}
Let $T$ be an $X$-tree, let $c$ be a leaf of $\CT$, $e(c)$ the edge connecting $c$ to $\CT$. 
Then,
$$
\mathscr M (T) = 
\mathscr M(T^{[c]})\oplus_{e(c)} \mathscr M(T_{[c]}). 
$$
\end{lemma}
\begin{proof} We check the definition via the set of flats. Fix $F\subseteq E$ and let $F':=F\cap E(T^{[c]})$, $F'':=F\cap T_{[c]}$. Then, $F$ is flat in $T$ if and only if $F'$ and $F''$ are both flat in the respective graphs (where unlabeled vertices have the same neighborhood as in $T$).
 \end{proof}

Recall that the {\em neighborhood} $N(v)$ of a vertex $v$ in a (simple) graph is the set of edges incident to $v$. The {\em degree} of $v$ is then the number $\deg(v):=\vert N(v)\vert$ of such edges. In the following, given any vertex $v$ of a tree $T$ and any $A\subseteq E(T)$ we will write $\deg_A(v)$ for the degree of $v$ in the graph $T[A]$. More generally, for any given set $W\subseteq V(T)$ of vertices, we write $\deg_A(W):=\sum_{v\in W}\deg_A(v)$. If no precision is necessary, we will write $\deg$ for $\deg_E$, the degree in the full graph $T$.

\begin{remark}\label{rem:CF}
If $c\in X$, then $\mathscr M(T_{[c]})\simeq \FM (\operatorname{deg}(c))$. Otherwise, $\mathscr M(T_{[c]})\simeq \mathscr C(\operatorname{deg}(c))$.
\end{remark}

\begin{theorem}
Let $T$ be an $X$-tree. Fix an enumeration $c_1,\ldots, c_m$ of the vertices of $\CT$ such that $\CT_i$, the graph induced on the vertices $c_1,\ldots, c_i$, is connected for all $i=1,\ldots,m$. Let $e_i$ denote the edge $\{c_i,c_{i+1}\}$.
Then,
$$
\mathscr M(T) = 
\mathscr R(\operatorname{deg}(c_1))
\oplus_{e_1}
\cdots
\oplus_{e_{m-1}}
\mathscr R(\operatorname{deg}(c_m))
$$
where $\mathscr R(\operatorname{deg}(c))$ is a matroid on the ground set $N(c)$ and
 equals  $\mathscr F(\operatorname{deg}(c))$ if $c\in X$ and equals $\mathscr C(\operatorname{deg}(c))$ otherwise.
\end{theorem}

\begin{proof}
A recursive application of Lemma \ref{lemma:pruning} gives
$$
\mathscr M(T) = \mathscr M(T^{[c_m]\cdots[c_2]}) \oplus_{e(c_2)} \mathscr M(T^{[c_m]\cdots[c_3]}_{[c_2]}) 
\ldots
\oplus_{e(c_{m-1})} 
\mathscr M(T_{[c_m]}).
$$
From this expansion the claim follows by Remark \ref{rem:CF}, noticing that $T^{[c_m]\cdots[c_{i}]}_{[c_{i-1}]} = T_{[c_{i-1}]}$ for all $i=2,\ldots,m$, and that $T^{[c_m]\cdots[c_2]}=T_{[c_1]}$.
\end{proof}

\def\deg{\operatorname{deg}}

\begin{definition}
Let $T$ be an $X$-tree as above.
We denote by $L$, resp.\ $U$, the set of labeled, resp.\ unlabeled vertices of the core $\CT$ (hence $L=V(\CT)\cap X$ and $V(\CT)=L\uplus U$).
 Moreover, given any $A\subseteq E$, let $\epsilon(A):= \vert A\cap E(\CT)\vert$ be the number of edges of $\CT$ contained in $A$, and write $\varphi (A)$ for the number of unlabeled vertices that are isolated in $T[A^c]$ (hence $\varphi(A)=\vert\{c\in U \mid \deg_A(c)=\deg(c)\}\vert$.
\end{definition}

\begin{lemma}\label{lem:RP}
The rank in $\mathscr M(T)$ of any $F\subseteq E$ is
\begin{equation}\label{eq:rank}
\deg_F(V(\CT)) - \varphi(F) -\epsilon(F)
\end{equation}
For any flat $F\subseteq E$, the characteristic polynomial of the contraction $\mathscr M(T)/F$ can be expressed as follows (where we write $G:=E\setminus F$).

\begin{equation}
(-1)^{\deg_{G}(V(\CT))-\epsilon(G)}
(1-t)^{\deg_{G}(L)-\epsilon(G)}
\prod_{\substack{c\in U\\ \deg_{G}(c)>0}}
\frac{(1-t)^{\deg_{G}(c)}-(1-t)}{t}
\end{equation}
\end{lemma}
\begin{proof}
The rank of $F$ is the sum of the ranks of $F\cap N(c)$ in $\mathscr R(c)$ to which one has to subtract the number of $e_i$ contained in $F$. The rank of $F\cap N(c)$ is its cardinality ($\deg_F(c)$), except in the case where $N(c)\subseteq F$ ($\deg_F(c)=\deg_T(c)$) and $c$ is unlabeled, where one has to substract one. The first claim follows.

For the second claim, repeated application of Remark \ref{rem:221}yields
$$
\chi_{\mathscr M(T)/F}(t)=
\frac{
\prod_{i=1}^m \chi_{\mathscr R(c_i)/(F\cap N(c_i))}(t)
}{
(t-1)^{\vert E(\CT)\setminus F\vert}
}
$$
For any $I\subseteq [k]$ with $i=\vert I \vert$ we have $\mathscr F(k)/I = \mathscr F(k-i)$ and $\mathscr C(k)/I = \mathscr C(k-i)$. Moreover, notice that $ \chi_{\mathscr C(0)}(t)=1$. With this, a direct computation proves the second claim.
\end{proof}

\begin{corollary}\label{cor:CP} 
Let $T$ be an $X$-tree. Then the characteristic polynomial of $\mathscr M(T)$ is as follows.
$$
\chi_{\mathscr M(T)}(t)=
(-1)^{\vert V(T)\vert-1}
(1-t)^{\deg(L)-\vert E(\CT) \vert }
\prod_{c\in U} 
\frac{(1-t)^{\deg(c)}-(1-t)}{t}
$$
The number of facets of the fundamental polytope $P_d$ (i.e., vertices of the Lipschitz polytope) is thus the following.
$$
2^{\deg(L) - \vert E(\CT)\vert }
\prod_{c\in U} (2^{\deg(c)}-2)
$$
\end{corollary}
\begin{proof} The first claim is immediate from Lemma \ref{lem:RP}, noticing that in an $X$-tree there are no unlabeled vertices of degree less than $3$. The second claim follows using Corollary \ref{eq:f:contr} and Equation \eqref{eq:f:contr} with $i=\vert X \vert - 1$.
\end{proof}

\begin{definition}
For a given $X$-tree $T$ and fixed $k\in \mathbb N^{V(\CT)}$, $i,\epsilon\in \mathbb N$, 
let  $\Gamma_T(k,\epsilon,i)$ denote the number of subgraphs $T[G]$ of $T$ with no unlabeled leaves and such that

\begin{itemize}
\item exactly $i+1$ connected components of $T[G]$ contain labeled vertices
\item $\deg_G(c)=k_c$ for all $c\in V(\CT)$.
 \item $G$ contains $\epsilon$ ''core edges'' ($\epsilon(G) = \epsilon$).
\end{itemize}
\end{definition}

\begin{remark}
Notice that only $G=E(T)$ satisfies the conditions when $i=0$; i.e., $\Gamma_T(k,\epsilon,0)=1$ if $k_c=\deg_T(c)$ for all $c$ and $\epsilon=\vert E(\CT)\vert$, and $0$ otherwise.
\end{remark}

\begin{theorem}\label{theo:formel}
Let $(X,d)$ be a tree-like metric space with associated $X$-tree $T$. Then, for all $i\geq 0$,
$$
f_i^{\LIP(X,d)} = f^{P_d}_{\vert X \vert - 1 - i} = 
\sum_{
(k,\epsilon)\in \mathbb N^{V(\CT) +1}
}
\Gamma_T(k,\epsilon,i)
2^{\sum_{c\in L}k_c -\epsilon }
\prod_{
\substack{c\in U\\ k_c>0}} (2^{k_c}-2)
$$
\end{theorem}

\begin{proof}
This formula is a direct consequence of Corollary \ref{eq:f:contr}, via Formula \eqref{eq:sum:contr} and the explicit expression of Lemma \ref{lem:RP}.
\end{proof}

\begin{example}\label{ss_Catg} Let $n\in \mathbb N$, $n\geq 3$, and let us consider any tree metric $(X,\gamma)$ whose underlying  $X$-tree is an $n$-caterpillar graph (see Figure \ref{fig_cg}) with every leaf labelled by exactly one of the $n$ points of $X$, and no internal vertices labelled. 

%\todo{scale}
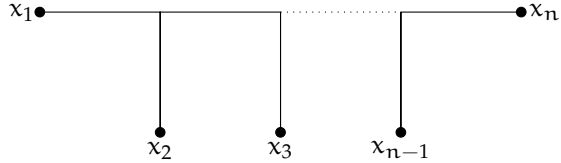
\begin{figure}[h]
\centering
\begin{tikzpicture}[scale=.8]
\draw   (-2,0)-- (0,0) -- (0,-2) -- (0,0) -- (2,0) -- (2,-2) -- (2,0);
\draw (4,0)-- (4,-2) -- (4,0)-- (6,0);
\draw[dotted](2,0) -- (4,0);

\draw[fill=black] (-2,0) circle(.8mm);
\draw[fill=black] (0,-2) circle(.8mm);
\draw[fill=black] (2,-2) circle(.8mm);
\draw[fill=black] (4,-2) circle(.8mm);
\draw[fill=black] (6,0) circle(.8mm);

\draw (-2.3,0) node {$x_1$};
\draw (0,-2.3) node {$x_2$};
\draw (2,-2.3) node {$x_3$};
\draw (4,-2.3) node {$x_{n-1}$};
\draw (6.4,0) node {$x_{n}$};
    
\end{tikzpicture}

\caption{The $n$-caterpillar graph}\label{fig_cg}
\end{figure}

Then, $\CT$ is an $(n-2)$-point path and our formula immediately computes the number of vertices of the associated Lipschitz-polytope as
$$
f_0^{\LIP(X,\gamma)}=
2\cdot 3^{n-2}.
$$

We can also compute the number of edges: notice that for $i=1$, the only subgraphs that are counted by $\Gamma_T(k,\epsilon,i)$ are of the form $T[E\setminus\{e\}]$ for some $e\in E$. 

If $e\not\in E(\CT)$ then $\epsilon=\vert E(\CT)\vert = n-3$ and the subgraph has exactly one unlabeled vertex of degree $2$, while all others have degree $3$. This results in a contribution to the number of edges in the amount of $2^{3-n}\cdot 2\cdot 6^{n-3}$ for each of these $n$ cases. If $e\in E(\CT)$, then $\epsilon=n-4$ and there are exactly two unlabeled vertices of degree two, while all others have degree $3$. This gives a contribution of $2^{4-n}\cdot 2^{2}\cdot 6^{n-4}$ in each of these $n-3$ cases. 

In total, the number of edges of the Lipschitz-polytope of the $n$-caterpillar graph then equals 
$$
f_1^{\LIP(X,\gamma)}=n2\cdot 3^{n-3}+ (n-3)2^2\cdot 3^{n-4}\textrm{ for all }n\geq 3.$$
These formulas do in fact correctly predict some of the numbers in Table \ref{table}, which shows the $f$-polynomials of the fundamental polytopes of these metric spaces for the first few values of $n$ as computed with SAGE via Corollary \ref{cor:flats}. It took around 10 seconds to compute the $f$-polynomial of the biggest example, the $6$-caterpillar graph, on the sage cloud (run on a free server).

\end{example}

\begin{table}[h]

\begin{tabular}{ l|l }

  Metric space  &  $f$--polynomial of $P_d(X)$\\
  \hline
3-caterpillar & $t^3+6t^2+6t+1$  \rule{0pt}{2.6ex}\\
4-caterpillar    & $t^4+12t^3+28t^2+18t+1$ \\
5-caterpillar    & $t^5+20t^4+80t^3+114t^2+54t+1$   \\
6-caterpillar     & $t^6+30t^5+180t^4+422t^3+432t^2+162t+1$   \\
\end{tabular}

%$\,$
\caption {} 
\label{table} 
\end{table}

\section{A characterization of simpliciality}

We turn to characterizing the tree-like metric spaces whose fundamental polytope is simplicial. Recall that a polytope is called simplicial if each of its faces is (combinatorially equivalent to) a simplex \cite[Section 2.5]{Ziegler}. Equivalently, a polytope $P$ is simplicial if, for every face $F\in \FF(P)$, the lower interval $\FF(P)_{\leq F}$ is a boolean poset.  Analogously, an arrangement $\AA$ of hyperplanes in a real vector space is called simplicial if each of the cones of the fan determined by $\AA$ (see \S \ref{ss:AH}) is a cone over a simplex - or, equivalently, if $\FF(\AA)_{\leq F}$ is a boolean poset for each $F\in \FF(\AA)$.

Consider a tree-like metric space $(X,d)$ with associated split system $\AA(\mathcal S)$. We know that the posets of faces of the fundamental polytope and of the associated hyperplane arrangement are isomorphic: $\FF(P_d(X)) \simeq \FF(\AA(\mathcal S))$. Therefore, our characterization of simpliciality for fundamental polytopes of metric trees will build upon the following characterization of simpliciality of arrangements of hyperplanes.

\begin{theorem}[Cuntz-Geis {\cite[Corollary 2.4]{CG}}]
Let $\mathscr A$ be an arrangement of hyperplanes in $\mathbb R^r$. Suppose that $\cap \mathscr {A}$ is a single point, so that the matroid $\mathscr{M}(\mathscr{A})$ has rank $r$.

The arrangement $\mathscr{A}$ is simplicial if and only if the characteristic polynomial satisfies
\begin{equation}\label{eq:CG}
r\,\chi_{\mathscr{M}(\mathscr{A})}(-1)+2\sum_{H\in \mathscr {A}} \chi_{\mathscr{M}(\mathscr A)/H}(-1)=0.
\end{equation}

\end{theorem}

\begin{lemma}\label{lem:q}
Let $T$ be an $X$-tree and $\mathscr M(T)$ the associated matroid. Let $e$ be an edge of $T$ and let $L(e)$, resp.\ $U(e)$, be the set of labeled, resp.\ unlabeled vertices of the edge $e$. Then
$$
\frac{\chi_{\mathscr M(T)/e}(-1)}{\chi_{\mathscr M(T)}(-1)}= -  2^{\vert U(e) \vert -1} \prod_{v\in U(e)}\frac{2^{\deg(v)-1}-2}{2^{\deg(v)}-2}
$$
\end{lemma}
\begin{proof}
The negative sign appears because the sign of the characteristic polynomial evaluated at $-1$ is the parity of the rank of the matroid, and contracting by a non-loop element decreases the matroid's rank by one. The formula for the absolute value follows from a case-by-case comparison using Lemma \ref{lem:RP}
\end{proof}

\begin{theorem}\label{theo:simp}
Let $(X,d)$ be a tree-like metric space. The fundamental polytope $P_d(X)$ is simplicial if and only if the space is a full tree. (In this case the face numbers are computed in Example \ref{ex:full}.)
\end{theorem}
\begin{proof}
With \Cref{lem:q}, and using the fact that $\mathscr{M}(\AA(\mathcal S))$ has rank $\vert X \vert -1$, Cuntz and Geis' condition (Equation \eqref{eq:CG}) is equivalent to 
\begin{equation}\label{eq:mod}
\left(
\vert X \vert - 1 - 2\sum_{e\in E(T)}q(e)
\right)
=0.
\end{equation}
where we write $q(e)$ for the absolute value of the quantity at the right-hand side of the claim in \Cref{lem:q}. Now notice that 
$$
q(e)=\frac{1}{2}\quad\textrm{ if }U(e)=\emptyset,
\quad\quad\quad
q(e)>\frac{1}{2}\quad\textrm{ otherwise, }
$$
since the degree of an unlabeled vertex in an $X$-tree is always at least $3$.
Thus the left-hand side of \Cref{eq:mod} is 

$$
\left(
\vert X \vert - 1 - 2\sum_{e\in E(T)}q(e)
\right)
\leq \vert X \vert - 1 -\vert E(T)\vert \leq 0
$$
where the first inequality is an equality if and only if  
$U(e)=\emptyset$ for all $e\in E(T)$. This means that no labeled vertices exist, hence $\vert X \vert = V(T)$ and the second inequality is also an equality because every finite tree has one more vertex than it has edges. The claim follows.
\end{proof}

\begin{definition}[Gordon and Petrov {\cite{Petrov}}]\label{df:generic}
A metric space $(X,d)$ is called {\em generic} if the triangle inequality is always strict (i.e., $d(x,z)<d(x,y)+d(y,z)$ for all pairwise distinct $x,y,z\in X$)  and the fundamental polytope $P_d(X)$ is simplicial.
\end{definition}

\begin{corollary}\label{cor:generic}
No tree-like metric on more than $2$ points is generic in the sense of Definition \ref{df:generic}.
\end{corollary}
\begin{proof}
Gordon and Petrov define a generic metric to be one that is always strict and for which the fundamental polytope is simplicial. We know by \Cref{theo:simp} that a tree metric for which the polytope is simplicial must be a full tree (i.e., without unlabeled vertices). But every full tree on $3$ or more vertices contains two adjacent edges, and the triangle inequality applied to the tree vertices incident to those two edges is in fact an equality. 
\end{proof}

 \bibliography{main2}

\begin{thebibliography}{10}

\bibitem{Dress}
H.-J. Bandelt and A.~W.~M. Dress.
\newblock A canonical decomposition theory for metrics on a finite set.
\newblock {\em Adv. Math.}, 92(1):47--105, 1992.

\bibitem{BdM}
J.~Bonin and A.~de~Mier.
\newblock Tutte polynomials of generalized parallel connections.
\newblock {\em Adv. Appl. Math.}, 32(1-2):31--43, 2004.

\bibitem{Buneman}
P.~Buneman.
\newblock A note on the metric properties of trees.
\newblock {\em J. Combinatorial Theory Ser. B}, 17:48--50, 1974.

\bibitem{CelliniMarietti}
P.~Cellini and M.~Marietti.
\newblock Root polytopes and {A}belian ideals.
\newblock {\em J. Algebraic Combin.}, 39(3):607--645, 2014.

\bibitem{CG}
M.~Cuntz and D.~Geis.
\newblock Combinatorial simpliciality of arrangements of hyperplanes.
\newblock {\em Beitr. Algebra Geom.}, 56(2):439--458, 2015.

\bibitem{DeStu}
M.~Develin and B.~Sturmfels.
\newblock Tropical convexity.
\newblock {\em Doc. Math.}, 9:1--27, 2004.

\bibitem{Deza}
M.~M. Deza and M.~Laurent.
\newblock {\em Geometry of Cuts and Metrics}.
\newblock Springer Publishing Company, Incorporated, 1st edition, 2009.

\bibitem{Petrov}
J.~Gordon and F.~Petrov.
\newblock Combinatorics of the {L}ipschitz polytope.
\newblock {\em Arnold Math. J.}, 3(2):205--218, 2017.

\bibitem{HJ}
S.~Herrmann and M.~Joswig.
\newblock Splitting polytopes.
\newblock {\em M\"unster J. Math.}, 1:109--141, 2008.

\bibitem{Michalek}
A.~Higashitani, K.~Jochemko, and M.~Micha\l{}ek.
\newblock Arithmetic aspects of symmetric edge polytopes.
\newblock {\em ArXiv e-prints}, 2018.

\bibitem{Hirai}
H.~Hirai.
\newblock A geometric study of the split decomposition.
\newblock {\em Discrete Comput. Geom.}, 36(2):331--361, 2006.

\bibitem{JJZ}
Filip~D. {Jevti{\'c}}, Marija {Jeli{\'c}}, and Rade~T. {{\v{Z}}ivaljevi{\'c}}.
\newblock {Cyclohedron and Kantorovich-Rubinstein polytopes}.
\newblock {\em arXiv e-prints}, Mar 2017.

\bibitem{JTZ}
Filip~D. {Jevti{\'c}}, Marinko {Timotijevi{\'c}}, and Rade~T.
  {{\v{Z}}ivaljevi{\'c}}.
\newblock {Polytopal Bier spheres and Kantorovich-Rubinstein polytopes of
  weighted cycles}.
\newblock {\em arXiv e-prints}, Dec 2018.

\bibitem{JL}
M.~Joswig and Georg L.
\newblock Weighted digraphs and tropical cones.
\newblock {\em Linear Algebra Appl.}, 501:304--343, 2016.

\bibitem{JS}
M.~Joswig and B.~Schr\"oter.
\newblock Matroids from hypersimplex splits.
\newblock {\em J. Combin. Theory Ser. A}, 151:254--284, 2017.

\bibitem{Koichi}
S.~Koichi.
\newblock The {B}uneman index via polyhedral split decomposition.
\newblock {\em Adv. in Appl. Math.}, 60:1--24, 2014.

\bibitem{Vershik2}
J.~Melleray, F.~Petrov, and A.~Vershik.
\newblock {Linearly rigid metric spaces and the embedding problem}.
\newblock {\em Fund. Math. 199,no.2 177-194}, 2008.

\bibitem{Murota}
K.~Murota.
\newblock {\em Discrete Convex Analysis: Monographs on Discrete Mathematics and
  Applications 10}.
\newblock Society for Industrial and Applied Mathematics, Philadelphia, PA,
  USA, 2003.

\bibitem{Giapponesi}
H.~Oshugi and T.~Hibi.
\newblock Normal polytopes arising from finite graphs.
\newblock {\em J. Algebra}, 207(2):409--426, 1998.

\bibitem{Oxley}
J.~Oxley.
\newblock {\em Matroid theory}, volume~21 of {\em Oxford Graduate Texts in
  Mathematics}.
\newblock Oxford University Press, Oxford, second edition, 2011.

\bibitem{Steel}
C.~Semple and M.~Steel.
\newblock {Phylogenetics}.
\newblock {\em Oxford University Press}, 2003.

\bibitem{Stanley}
R.~P. Stanley.
\newblock {\em Enumerative combinatorics. {V}olume 1}, volume~49 of {\em
  Cambridge Studies in Advanced Mathematics}.
\newblock Cambridge University Press, Cambridge, second edition, 2012.

\bibitem{StuYu}
B.~Sturmfels and J.~Yu.
\newblock Classification of six-point metrics.
\newblock {\em Electron. J. Combin.}, 11(1):Research Paper 44, 16, 2004.

\bibitem{Vershik1}
A.~M. Vershik.
\newblock Classification of finite metric spaces and combinatorics of convex
  polytopes.
\newblock {\em Arnold Math. J.}, 1(1):75--81, 2015.

\bibitem{Wu}
Y.~Wu, Z.~Xu, and Y.~Zhu.
\newblock Average range of {L}ipschitz functions on trees.
\newblock {\em Mosc. J. Comb. Number Theory}, 6(1):96--116, 2016.

\bibitem{Zaslavsky}
T.~Zaslavsky.
\newblock {\em {F}acing up to arrangements: face-count formulas for partitions
  of space by hyperplanes}.
\newblock ProQuest LLC, Ann Arbor, MI, 1974.
\newblock Thesis (Ph.D.)--Massachusetts Institute of Technology.

\bibitem{Ziegler}
G.~M. Ziegler.
\newblock {\em Lectures on polytopes}, volume 152 of {\em Graduate Texts in
  Mathematics}.
\newblock Springer-Verlag, New York, 1995.

\end{thebibliography}
 \bibliographystyle{plain}

\end{document}